\PassOptionsToPackage{unicode}{hyperref}
\PassOptionsToPackage{hyphens}{url}
\documentclass[
  12pt,
]{article}
\usepackage{amsmath,amssymb}
\usepackage{iftex}
\ifPDFTeX
  \usepackage[T1]{fontenc}
  \usepackage[utf8]{inputenc}
  \usepackage{textcomp} 
\else 
  \usepackage{unicode-math} 
  \defaultfontfeatures{Scale=MatchLowercase}
  \defaultfontfeatures[\rmfamily]{Ligatures=TeX,Scale=1}
\fi
\usepackage{lmodern}
\ifPDFTeX\else
\fi
\IfFileExists{upquote.sty}{\usepackage{upquote}}{}
\IfFileExists{microtype.sty}{
  \usepackage[]{microtype}
  \UseMicrotypeSet[protrusion]{basicmath} 
}{}
\makeatletter
\@ifundefined{KOMAClassName}{
  \IfFileExists{parskip.sty}{%
    \usepackage{parskip}
  }{
    \setlength{\parindent}{0pt}
    \setlength{\parskip}{6pt plus 2pt minus 1pt}}
}{
  \KOMAoptions{parskip=half}}
\makeatother
\usepackage{xcolor}
\usepackage[margin=1in]{geometry}
\usepackage{color}
\usepackage{fancyvrb}

\DefineVerbatimEnvironment{Highlighting}{Verbatim}{commandchars=\\\{\}}
\usepackage{framed}
\definecolor{shadecolor}{RGB}{248,248,248}
\newenvironment{Shaded}{\begin{snugshade}}{\end{snugshade}}

\newcommand{\AttributeTok}[1]{\textcolor[rgb]{0.13,0.29,0.53}{#1}}

\newcommand{\ConstantTok}[1]{\textcolor[rgb]{0.56,0.35,0.01}{#1}}
\newcommand{\ControlFlowTok}[1]{\textcolor[rgb]{0.13,0.29,0.53}{\textbf{#1}}}

\newcommand{\DecValTok}[1]{\textcolor[rgb]{0.00,0.00,0.81}{#1}}

\newcommand{\FloatTok}[1]{\textcolor[rgb]{0.00,0.00,0.81}{#1}}
\newcommand{\FunctionTok}[1]{\textcolor[rgb]{0.13,0.29,0.53}{\textbf{#1}}}

\newcommand{\NormalTok}[1]{#1}

\newcommand{\OtherTok}[1]{\textcolor[rgb]{0.56,0.35,0.01}{#1}}

\newcommand{\SpecialCharTok}[1]{\textcolor[rgb]{0.81,0.36,0.00}{\textbf{#1}}}

\newcommand{\StringTok}[1]{\textcolor[rgb]{0.31,0.60,0.02}{#1}}

\usepackage{longtable,booktabs,array}
\usepackage{calc} 
\usepackage{etoolbox}
\makeatletter
\patchcmd\longtable{\par}{\if@noskipsec\mbox{}\fi\par}{}{}
\makeatother
\IfFileExists{footnotehyper.sty}{\usepackage{footnotehyper}}{\usepackage{footnote}}
\makesavenoteenv{longtable}
\usepackage{graphicx}
\makeatletter
\newsavebox\pandoc@box
\newcommand*\pandocbounded[1]{
  \sbox\pandoc@box{#1}%
  \Gscale@div\@tempa{\textheight}{\dimexpr\ht\pandoc@box+\dp\pandoc@box\relax}%
  \Gscale@div\@tempb{\linewidth}{\wd\pandoc@box}%
  \ifdim\@tempb\p@<\@tempa\p@\let\@tempa\@tempb\fi
  \ifdim\@tempa\p@<\p@\scalebox{\@tempa}{\usebox\pandoc@box}%
  \else\usebox{\pandoc@box}%
  \fi%
}
\def\fps@figure{htbp}
\makeatother
\NewDocumentCommand\citeproctext{}{}

\makeatletter
 \let\@cite@ofmt\@firstofone
 \def\@biblabel#1{}
 \def\@cite#1#2{{#1\if@tempswa , #2\fi}}
\makeatother
\newlength{\cslhangindent}
\newlength{\csllabelwidth}
\newenvironment{CSLReferences}[2] 
 {\begin{list}{}{%
  \setlength{\itemindent}{0pt}
  \setlength{\leftmargin}{0pt}
  \setlength{\parsep}{0pt}
  \ifodd #1
   \setlength{\leftmargin}{\cslhangindent}
   \setlength{\itemindent}{-1\cslhangindent}
  \fi
  \setlength{\itemsep}{#2\baselineskip}}}
 {\end{list}}
\usepackage{calc}

\usepackage{tcolorbox}
\usepackage{amssymb}
\usepackage{yfonts}
\usepackage{bm}
\usepackage{titlesec}

\newtcolorbox{greybox}{
  colback=white,
  colframe=blue,
  coltext=black,
  boxsep=5pt,
  arc=4pt}

\usepackage{bookmark}
\IfFileExists{xurl.sty}{\usepackage{xurl}}{} 
\hypersetup{
  pdftitle={Majorizing Stress Formula Two},
  pdfauthor={Jan de Leeuw - University of California Los Angeles},
  hidelinks,
  pdfcreator={LaTeX via pandoc}}

\title{Majorizing Stress Formula Two}
\author{Jan de Leeuw - University of California Los Angeles}
\date{Started April 19 2024, Version of July 25, 2024}

\usepackage{amsthm}
\newtheorem{theorem}{Theorem}[section]
\newtheorem{lemma}{Lemma}[section]

\theoremstyle{definition}

\theoremstyle{definition}

\theoremstyle{definition}

\theoremstyle{definition}

\theoremstyle{remark}

\begin{document}
\maketitle
\begin{abstract}
Modifications of the smacof algorithm for multidimensional scaling are proposed that provide a convergent majorization algorithm for Kruskal's stress formula two.
\end{abstract}

{
\setcounter{tocdepth}{3}
\tableofcontents
}
\textbf{Note:} This is a working paper which will be expanded/updated frequently. All suggestions for improvement are welcome.

\section{Introduction}\label{introduction}

The loss function minimized in the current non-metric and non-linear R implementations of the smacof programs for MDS
(De Leeuw and Mair (2009), Mair, Groenen, and De Leeuw (2022)) is Kruskal's original \emph{normalized stress} (Kruskal (1964a), Kruskal (1964b)). It is defined as
\begin{equation}
\sigma_1(X):=\frac{\sum\sum w_{ij}(\delta_{ij}- d_{ij}(X)^2)}{\sum\sum w_{ij}d_{ij}^2(X)}.
\label{eq:stress1def}
\end{equation}
In equation \eqref{eq:stress1def} we assume throughout that dissimilarities \(\delta_{ij}\)
and weights \(w_{ij}\) are non-negative, and, without loss of generality, that the weights add up to one. The double summation is over all pairs of indices \((i,j)\) with \(i>j\), i.e, over the
elements below the diagonal of the matrices \(\Delta\), \(W\), and \(D(X)\).

In Kruskal (1965) a different loss function was used in the context of using monotone transformations when fitting a linear model. In MDS this loss function is
\begin{equation}
\sigma_2(X):=\frac{\sum\sum w_{ij}(\delta_{ij}- d_{ij}(X)^2)}{\sum\sum w_{ij}(d_{ij}(X)-\overline{d}(X))^2},
\label{eq:stress2def}
\end{equation}
where
\begin{equation}
\overline{d}(X)=\sum\sum w_{ij}d_{ij}(X).
\label{eq:doverdef}
\end{equation}

In Kruskal and Carroll (1969), in the section written by Kruskal (p.~652), we see

\begin{quote}
In several of my scaling programs, I refer to these expressions
as ``stress formula one'' and ``stress formula two'', respectively.
Historically, stress formula one was the only badness-of-fit
function used for some time. Stress formula two has been used
more recently and I now tend to recommend it.
\end{quote}

Another early adopter (Roskam (1968), p.~34) says

\begin{quote}
While the original formula is adequate for completely ordered
B-data, we found it is not adequate with completely ordered
A-data.
\end{quote}

The distinction between A-data and B-data comes from Coombs (1964).
For B-data the \(\delta_{ij}\) are dissimilarties
between pairs of elements of a single set, while for A-data they
are dissimilarities between two different sets, a row-set and a column-set.
Moreover both Kruskal and Roskam found that having the variance
of the distances in the denominator of stress has major advantages
for conditional A-data, in which only comparisons of dissimilarities
with in the same row are meaningful.

In this paper we will extend the theory and algorithm of smacof
to stress formula two. We emphasize that normalized loss functions
such as \(\sigma_1\) and \(\sigma_2\) should are only
used in non-linear or nor-metric MDS problems. In metric MDS problems
raw stress, without any normalization, can be used.

\section{Problem}\label{problem}

We want to minimize Kruskal's \(\sigma_2\) from \eqref{eq:stress2def}
over the \(n\times p\) configuration matrices \(X\).

It is convenient to have some notation for the numerator and denominator of the two stress formulas.
\begin{subequations}
\begin{align}
\sigma_R(X)&:=\sum\sum w_{ij}(\delta_{ij}-d_{ij}(X))^2,\label{eq:rawdef}\\
\eta_1^2(X)&:=\sum\sum w_{ij}d_{ij}^2(X),\label{eq:eta1def}\\
\eta_2^2(X)&:=\sum\sum w_{ij}(d_{ij}(X)-\overline{d}(X))^2,\label{eq:eta2def}
\end{align}
\end{subequations}
Kruskal terms \(\sigma_R\) from definition \eqref{eq:rawdef} the \emph{raw stress}.

There have not been any systematic comparisons of the two stress formulas, and
the solutions they lead to, that I am aware of.
Kruskal (in Kruskal and Carroll (1969), p.~652) says

\begin{quote}
For any given configuration, of course, stress formula two yields
a substantially larger value than stress formula one, perhaps twice
as large in many cases. However, in typical multidimensional scaling
applications, minimizing stress formula two typically yields very similar
configurations to minimizing stress formula one.
\end{quote}

We can get some idea about the difference in scale of the two loss functions from the results
\begin{equation}
\frac{\sigma_1(X)}{\sigma_2(X)}=\frac{\eta_2^2(X)}{\eta_1^2(X)}\geq\min_X\frac{\eta_2^2(X)}{\eta_1^2(X)}
\label{eq:compa}
\end{equation}
De Leeuw and Stoop (1984) show that in the one-dimensional case with \(p=1\) and with
all \(w_{ij}\) equal, this implies
\begin{equation}
\sigma_1(X)\geq\frac13\frac{n-2}{n}\sigma_2(X).
\label{eq:bound}
\end{equation}
Thus in this special case \(\sigma_1\) is three to nine times as large as
\(\sigma_2\). In general the bound in equation \eqref{eq:bound} depends on the weights, on the dimensionality \(p\), and on the order \(n\) of the problem.

As a qualitative statement, supported to some extent by the computations of De Leeuw and Stoop (1984),
we can say that minimizing \(\sigma_1\) will
tend to give optimal configurations in which distances have less variance than those in
configurations that minimize \(\sigma_2\). One thing is for sure, however. If \(X\)
is a regular simplex in \(n-1\) dimensions then \(\sigma_2\) is not even defined.
Or, to put it differently, if all \(\delta_{ij}\) are equal the minimum of
\(\sigma_2\) in \(n-1\) dimensions does not exist.

\section{Notation}\label{notation}

Now for some notation. As in standard MDS theory (De Leeuw (1977), De Leeuw and Heiser (1977), De Leeuw (1988)) we use the matrices
\begin{equation}
A_{ij}:=(e_i-e_j)(e_i-e_j)',
\label{eq:adef}
\end{equation}
where \(e_i\) are unit vectors with element \(i\) equal to one and the other \(n-1\) elements equal to zero.
Thus \(A_{ij}\) has elements \((i,i)\) and \((j,j)\) equal to \(+1\), elements \((i,j)\) and \((j,i)\) equal to
\(-1\), and all other elements equal to zero. The usefulness of the \(A_{ij}\) in MDS derives mainly from the
formula
\begin{equation}
d_{ij}^2(X)=\text{tr}\ X'A_{ij}X.
\label{eq:d2froma}
\end{equation}

Using the \(A_{ij}\) we now define other matrices, also standard in MDS,
\begin{subequations}
\begin{align}
V&:=\sum\sum w_{ij}A_{ij},\label{eq:vdef}\\
B(X)&:=\sum\sum w_{ij}\frac{\delta_{ij}}{d_{ij}(X)}A_{ij}.\label{eq:bdef}
\end{align}
\end{subequations}
Note that \(B\) is a matrix-valued function, not a single matrix. For completeness
also define
\begin{equation}
\eta^2(\Delta):=\sum\sum w_{ij}\delta_{ij}^2.
\label{eq:etadeltadef}
\end{equation}

Specifically because we are dealing with \(\sigma_2\) we also need the non-standard
definition
\begin{equation}
M(X):=\overline{d}(X)\sum\sum\frac{w_{ij}}{d_{ij}(X)}A_{ij}.
\label{eq:mdef}
\end{equation}

In both definitions \eqref{eq:bdef} and \eqref{eq:mdef} the summation is
over pairs \((i,j)\) with \(d_{ij}(X)>0\). Of course we can also omit all
pairs from the summation for which \(w_{ij}=0\).

\section{Majorization}\label{majorization}

In this section we construct a convergent majorization algorithm (De Leeuw (1994))
(also known as an MM algorithm, Lange (2016)) to minimize \(\sigma_2\).

The first step is to turn the minimization of a ratio of two functions into the
iterative minimization of a difference of the two functions. This is
a classical trick in fractional programming, usually attributed to
Dinkelbach (1967). Define
\begin{equation}
\omega(X,Y):=\sum\sum w_{ij}(\delta_{ij} - d_{ij}(X))^2-\sigma(Y)\{\sum\sum w_{ij}(d_{ij}(X)-\overline{d}(X))^2\}
\label{eq:omegadef}
\end{equation}

\begin{lemma}
\protect\hypertarget{lem:dinkelbach}{}\label{lem:dinkelbach}If \(\omega(X,Y)<\omega(Y,Y)=0\) then \(\sigma(X)<\sigma(Y)\).
\end{lemma}

\begin{proof}
This is embarassingly simple.
Direct substitution shows \(\omega(X,X)=0\) for all \(X\). Also \(\omega(X,Y)<0\) if and only if
\begin{equation}
\sum\sum w_{ij}(\delta_{ij} - d_{ij}(X))^2<\sigma(Y)\{\sum\sum w_{ij}(d_{ij}(X)-\overline{d}(X))^2\}
\label{eq:dinkelbach}
\end{equation}
Dividing both sides by \(\{\sum\sum w_{ij}(d_{ij}(X)-\overline{d}(X))^2\}\) shows that \(\sigma(X)<\sigma(Y)\).
\end{proof}

It follows from lemma \ref{lem:dinkelbach} that if we are in iteration \(k\), with tentative solution \(X^{(k)}\), then finding any \(X^{(k+1)}\) such that
\(\omega(X^{(k+1)},X^{(k)})<0\) will decrease stress. We will accomplish this
in our algorithm by performing one or more majorization steps decreasing
\(\omega(X,X^{(k)})\).

We should note that as a general strategy we cannot use finding \(X^{(k+1)}\)
by minimizing \(\omega(X,X^{(k)})\) over \(X\). If the minimum exists this will
work, but in general \(\omega(\bullet,X^{(k)})\) may be unbounded below, and
the minimum may not exist. This is easily seen from the example \(f(x)=x'Ax/x'x\)
for which Dinkelbach's maneuver gives \(g(x,y)=x'Ax-f(y)x'x\). The minimum
of \(g\) over \(x\) is zero if \(f(y)\) is equal to \(\lambda_{\min}(A)\), the smallest eigenvalue
of \(A\), which is actually the minimum of \(f\). If \(f(y)>\lambda_{min}(A)\) the minimum
does not exist (the infimum is \(-\infty\)). We can ignore the case \(f(y)<\lambda_{min}(A)\).
because that is impossible. But if \(f(y)>\lambda_{min}(A)\) any \(x\) with \(x'x=1\) other than
the eigenvector corresponding with the minimum eigenvalue satisfies \(g(x,y)<0\)
and thus \(f(x)<f(y)\).

Back to \(\sigma_2\). From definitions \eqref{eq:vdef}, \eqref{eq:bdef}, \eqref{eq:etadeltadef}, and \eqref{eq:mdef}
\begin{equation}
\omega(X,Y)=\eta^2(\Delta)+(1-\sigma(Y))\text{tr}\ X'VX-2\ \text{tr}\ X'B(X)X+\text{tr}\ X'M(X)X
\label{eq:omegasimple}
\end{equation}

\begin{lemma}
\protect\hypertarget{lem:smacof}{}\label{lem:smacof}For all \(X\) and \(Y\)
\begin{equation}
\text{tr}\ X'B(X)X\geq\text{tr}\ X'B(Y)Y,
\label{eq:smacof}
\end{equation}
with equality if \(X=Y\).
\end{lemma}

\begin{proof}
By Cauchy-Schwartz
\begin{equation}
d_{ij}(X)\geq\frac{1}{d_{ij}(Y)}\text{tr}\ X'A_{ij}Y
\label{eq:cs1}
\end{equation}
Multiplying both sides by \(w_{ij}\delta_{ij}\) and summing proves the lemma.
\end{proof}

\begin{lemma}
\protect\hypertarget{lem:neweq}{}\label{lem:neweq}For all \(X\) and \(Y\)
\begin{equation}
\text{tr}\ X'M(X)X\leq\text{tr}\ X'M(Y)X,
\label{eq:secineq}
\end{equation}
with equality if \(X=Y\).
\end{lemma}

\begin{proof}
Start with the trivial result
\begin{equation}
\sum\sum w_{ij}d_{ij}(X)=\sum\sum \frac{w_{ij}}{d_{ij}(Y)}d_{ij}(X)d_{ij}(Y).
\label{eq:trivial}
\end{equation}
By Cauchy-Schwartz
\begin{equation}
\overline{d}(X)\leq\sqrt{\sum\sum \frac{w_{ij}}{d_{ij}(Y)}d_{ij}^2(X)}\sqrt{\sum\sum \frac{w_{ij}}{d_{ij}(Y)}d_{ij}^2(Y)}
\label{eq:cs2}
\end{equation}
Squaring both sides proves the lemma.
\end{proof}

We are now ready for the main result.

\begin{theorem}
\protect\hypertarget{thm:main}{}\label{thm:main}Suppose \(\sigma_2(X^{(0)})\leq 1\). The update\\
\begin{equation}
X^{(k+1)}=\{(1-\sigma_2(X^{(k)}))V+\sigma_2(X^{(k)})M(X^{(k)})\}^+B(X^{(k)})X^{(k)}
\label{eq:update}
\end{equation}
defines a convergent majorization algorithm.
\end{theorem}

\begin{proof}
Using the definitions in equations \eqref{eq:vdef}, \eqref{eq:bdef}, \eqref{eq:etadeltadef}, and \eqref{eq:mdef} define
\begin{equation}
\xi(X,Y):=\eta^2(\Delta)+(1-\sigma(Y))\text{tr}\ X'VX-2\text{tr}\ X'B(Y)Y+\sigma(Y)\text{tr}\ X'M(Y)X.
\label{eq:xidef}
\end{equation}
From lemmas \ref{lem:smacof} and \ref{lem:neweq} \(\omega(X,Y)\leq\xi(X,Y)\)
with equality if \(X=Y\). In particular
\begin{subequations}
\begin{equation}
\omega(X^{(k+1)},X^{(k)})\leq\xi(X^{(k+1)},X^{(k)}).
\label{eq:ineq1}
\end{equation}
The update $X^{(k+1)}$ minimizes $\xi(X,X^{(k)})$
and thus 
\begin{equation}
\xi(X^{(k+1)},X^{(k)})\leq\xi(X^{(k)},X^{(k)})=\omega(X^{(k)},X^{(k)}).
\label{eq:ineq2}
\end{equation}
\end{subequations}
Combining equations \eqref{eq:ineq1} and \eqref{eq:ineq2}, and using lemma \ref{lem:dinkelbach}, shows that also \(\sigma_2(X^{(k+1)})\leq\sigma_2(X^{(k)})\).
\end{proof}

In order to make our proof work we had to guarantee that for all \(k\)
\begin{equation}
(1-\sigma_2(X^{(k)}))V+\sigma_2(X^{(k)})M(X^{(k)})\gtrsim 0,
\label{eq:psd}
\end{equation}
because otherwise the minimum of \(\xi(\bullet,X^{(k)})\) does not exist. If \(\sigma_2(X^{(k)})\leq 1\) the matrix in inequality \eqref{eq:psd} is a convex combination of two positive semi-definite matrices, and is thus
positive semi-definite. And because of theorem \ref{thm:main} it is sufficient to assume that
\(\sigma_2(X^{(0)})\leq 1\), because subsequent \(X^{(k)}\) will have \(\sigma_2\) values
smaller than the value for \(X^{(0)}\). Thus we need to start our majorization algorithm with
a sufficiently good initial estimate of \(X\). A random start may not work.

From the practical point of view the condition \(\sigma_2(X^{(0)})\leq 1\) is not
really restrictive. As the introduction of this paper says, in metric MDS we do not
use \(\sigma_2\). But even in metric MDS the Torgerson initial estimate usually
takes \(\sigma_2\) well below one. In non-linear or non-metric scaling the \(\delta_{ij}\) are optimal transformations or quantifications. If the optimum transformation is better
than the optimal constant transformation the condition \(\sigma_2(X^{(k)})\leq 1\)
is automatically satisfied for all \(k\). And even if the optimum transformation is
the constant transformation we still have \(\sigma_2(X^{(k)})=1\) and
inequality \eqref{eq:psd} is satisfied.

If for some reason you want to proceed if the matrix in \eqref{eq:psd} is
not positive semi-definite, then it suffices to choose any \(Y\) with
\begin{equation}
\text{tr}\ Y'\{((1-\sigma_2(X^{(k)}))V+\sigma_2(X^{(k)})M(X^{(k)}))\}Y\geq 0
\label{eq:emergency}
\end{equation}
and to minimize \(\xi(X^{(k)}+\alpha Y,X^{(k)})\) over \(\alpha\).

\section{Derivatives}\label{derivatives}

The derivatives of \(\sigma_2\) are
\begin{equation}
\mathcal{D}\sigma_2(X)=\frac{\mathcal{D}\sigma_R(X)-\sigma_2(X)\mathcal{D}\eta^2_2(X)}{
\eta_2^2(X)}
\label{eq:s2deriv}
\end{equation}
Now
\begin{subequations}
\begin{align}
\mathcal{D}\sigma_R(X)&=-2\sum\sum w_{ij}(\delta_{ij}-d_{ij}(X))\mathcal{D}d_{ij}(X),\\
\mathcal{D}\eta_2^2(X)&=2\sum\sum w_{ij}\mathcal{D}d_{ij}^2(X) -2\overline{d}(X)\sum\sum w_{ij}\mathcal{D}d_{ij}(X),
\end{align}
\end{subequations}
and
\begin{equation}
\mathcal{D}d_{ij}(X)=\frac{1}{d_{ij}(X)}A_{ij}X.
\label{eq:dderiv}
\end{equation}
And thus, using definitions \eqref{eq:vdef}, \eqref{eq:bdef}, and \eqref{eq:mdef}
\begin{subequations}
\begin{align}
\mathcal{D}\sigma_R(X)&=2(V-B(X))X,\\
\mathcal{D}\eta_2^2(X)&=2(V-M(X))X.
\end{align}
\end{subequations}

It follows that \(\mathcal{D}\sigma_2(X)=0\) if and only if
\begin{equation}
X=\{(1-\sigma_2(X))V+\sigma_2(X)M(X)\}^+B(X)X.
\label{eq:stationary}
\end{equation}
We can summarize the results of our computations in this section.

\begin{theorem}
\(X\) is a fixed point of the majorization iterations \eqref{eq:update} if and only if \(\mathcal{D}\sigma_2(X)=0\).
\end{theorem}

\section{Examples}\label{examples}

Although we mentioned in the introduction that it is unusual to use
\(\sigma_2\) in metric MDS problems we will nevertheless give some metric
examples to illustrate the algorithm. In both examples we start with
the Torgerson initial solution which takes the initial \(\sigma_2\) way below
one.

\subsection{Ekman}\label{ekman}

Our first example are the obligatory color data from Ekman (1954). The stress2
program produces the following sequence of \(\sigma_2\) values and
converges in 28 iterations.

\begin{verbatim}
## itel  1 sold  0.1577255150 snew  0.1321216983 
## itel  2 sold  0.1321216983 snew  0.1207395499 
## itel  3 sold  0.1207395499 snew  0.1156260670 
## itel  4 sold  0.1156260670 snew  0.1135043532 
## itel  5 sold  0.1135043532 snew  0.1126543441 
## itel  6 sold  0.1126543441 snew  0.1123159800 
## itel  7 sold  0.1123159800 snew  0.1121798388 
## itel  8 sold  0.1121798388 snew  0.1121239038 
## itel  9 sold  0.1121239038 snew  0.1121002964 
## itel  10 sold  0.1121002964 snew  0.1120900307 
## itel  11 sold  0.1120900307 snew  0.1120854276 
## itel  12 sold  0.1120854276 snew  0.1120833009 
## itel  13 sold  0.1120833009 snew  0.1120822904 
## itel  14 sold  0.1120822904 snew  0.1120817979 
## itel  15 sold  0.1120817979 snew  0.1120815523 
## itel  16 sold  0.1120815523 snew  0.1120814273 
## itel  17 sold  0.1120814273 snew  0.1120813627 
## itel  18 sold  0.1120813627 snew  0.1120813287 
## itel  19 sold  0.1120813287 snew  0.1120813107 
## itel  20 sold  0.1120813107 snew  0.1120813010 
## itel  21 sold  0.1120813010 snew  0.1120812957 
## itel  22 sold  0.1120812957 snew  0.1120812929 
## itel  23 sold  0.1120812929 snew  0.1120812913 
## itel  24 sold  0.1120812913 snew  0.1120812904 
## itel  25 sold  0.1120812904 snew  0.1120812899 
## itel  26 sold  0.1120812899 snew  0.1120812897 
## itel  27 sold  0.1120812897 snew  0.1120812895 
## itel  28 sold  0.1120812895 snew  0.1120812894
\end{verbatim}

The optimum configuration is in figure \ref{fig:ekfig1}, which can be compared with the
solution minimizing raw stress (which is identical up to a scale factor with the solution
minimizing stress formula one) in figure \ref{fig:ekfig2}. The raw stress solution
reaches stress formula one equal to 0.5278528 in 32 iterations. The two optimal
configurations are virtually identical.

\begin{figure}

{\centering \includegraphics{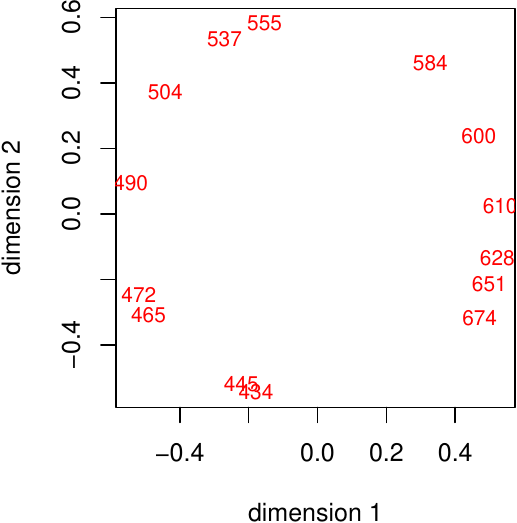} 

}

\caption{Ekman Metric Stress 2 Solution}\label{fig:ekfig1}
\end{figure}

\begin{figure}

{\centering \includegraphics{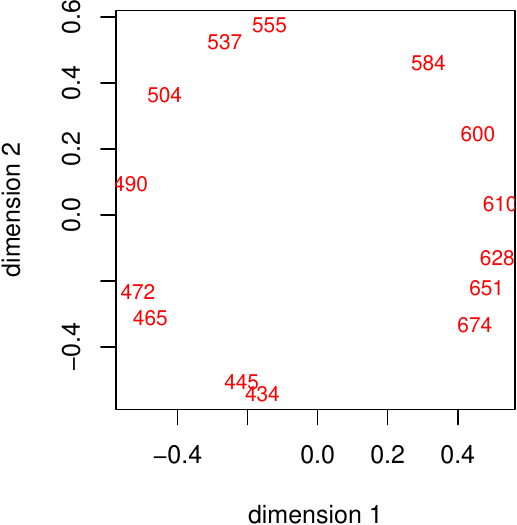} 

}

\caption{Ekman Metric Raw Stress Solution}\label{fig:ekfig2}
\end{figure}

\subsection{De Gruijter}\label{de-gruijter}

The Ekman data have an excellent fit in two dimensions and the optimum configuration is
extremely stable over variations in MDS methods. The data from De Gruijter (1967) on the
similarities between nine Dutch political parties in 1966 have a worse fit, and
less stability.

The solution minimizing \(\sigma)2\) has a loss of 0.3482919 and uses 230 iterations.
Minimizing raw stress finds stress 9.4408856 and uses 244 iterations. The optimal
configurations in figures \ref{fig:grfig1} and \ref{fig:grfig2} are similar, but definitely
not the same. Specifically the position of D66 (a ``pragmatic'' party, ideologically neither left nor right, established only in 1966, i.e.~in the year of the De Gruijter study) differs a lot between solutions.

\begin{figure}

{\centering \includegraphics{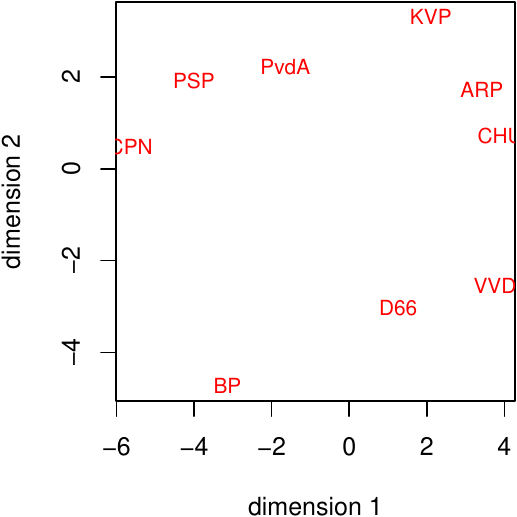} 

}

\caption{Gruijter Metric Stress 2 Solution}\label{fig:grfig1}
\end{figure}

\begin{figure}

{\centering \includegraphics{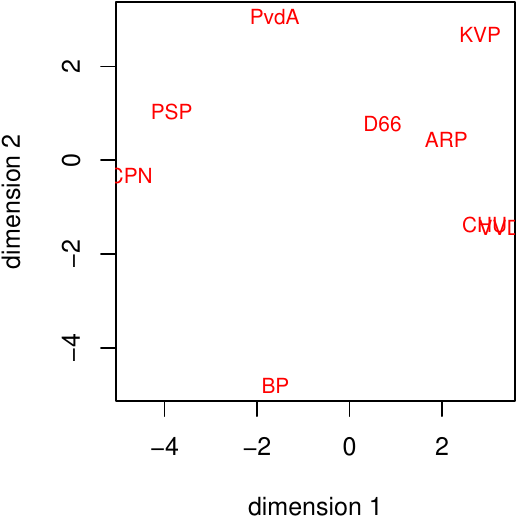} 

}

\caption{Gruijter Metric Raw Stress Solution}\label{fig:grfig2}
\end{figure}

\section{Appendix: Code}\label{appendix-code}

\subsection{stress2.R}\label{stress2.r}

\begin{Shaded}
\begin{Highlighting}[]
\NormalTok{stress2 }\OtherTok{\textless{}{-}}
  \ControlFlowTok{function}\NormalTok{(delta,}
           \AttributeTok{wmat =} \DecValTok{1} \SpecialCharTok{{-}} \FunctionTok{diag}\NormalTok{(}\FunctionTok{nrow}\NormalTok{(delta)),}
           \AttributeTok{ndim =} \DecValTok{2}\NormalTok{,}
           \AttributeTok{itmax =} \DecValTok{1000}\NormalTok{,}
           \AttributeTok{eps =} \FloatTok{1e{-}10}\NormalTok{,}
           \AttributeTok{verbose =} \ConstantTok{TRUE}\NormalTok{) \{}
\NormalTok{    itel }\OtherTok{\textless{}{-}} \DecValTok{1}
\NormalTok{    n }\OtherTok{\textless{}{-}} \FunctionTok{nrow}\NormalTok{(delta)}
\NormalTok{    wmat }\OtherTok{\textless{}{-}}\NormalTok{ wmat }\SpecialCharTok{/} \FunctionTok{sum}\NormalTok{(wmat)}
\NormalTok{    vmat }\OtherTok{\textless{}{-}} \SpecialCharTok{{-}}\NormalTok{wmat}
    \FunctionTok{diag}\NormalTok{(vmat) }\OtherTok{\textless{}{-}} \SpecialCharTok{{-}}\FunctionTok{rowSums}\NormalTok{(vmat)}
\NormalTok{    xold }\OtherTok{\textless{}{-}} \FunctionTok{torgerson}\NormalTok{(delta, ndim)}
\NormalTok{    dold }\OtherTok{\textless{}{-}} \FunctionTok{as.matrix}\NormalTok{(}\FunctionTok{dist}\NormalTok{(xold))}
\NormalTok{    enum }\OtherTok{\textless{}{-}} \FunctionTok{sum}\NormalTok{(wmat }\SpecialCharTok{*}\NormalTok{ delta }\SpecialCharTok{*}\NormalTok{ dold)}
\NormalTok{    eden }\OtherTok{\textless{}{-}} \FunctionTok{sum}\NormalTok{(wmat }\SpecialCharTok{*}\NormalTok{ dold }\SpecialCharTok{\^{}} \DecValTok{2}\NormalTok{)}
\NormalTok{    lbda }\OtherTok{\textless{}{-}}\NormalTok{ enum }\SpecialCharTok{/}\NormalTok{ eden}
\NormalTok{    dold }\OtherTok{\textless{}{-}}\NormalTok{ lbda }\SpecialCharTok{*}\NormalTok{ dold}
\NormalTok{    xold }\OtherTok{\textless{}{-}}\NormalTok{ lbda }\SpecialCharTok{*}\NormalTok{ xold}
\NormalTok{    aold }\OtherTok{\textless{}{-}} \FunctionTok{sum}\NormalTok{(wmat }\SpecialCharTok{*}\NormalTok{ dold)}
\NormalTok{    sold }\OtherTok{\textless{}{-}} \FunctionTok{sum}\NormalTok{(wmat }\SpecialCharTok{*}\NormalTok{ (delta }\SpecialCharTok{{-}}\NormalTok{ dold) }\SpecialCharTok{\^{}} \DecValTok{2}\NormalTok{) }\SpecialCharTok{/} \FunctionTok{sum}\NormalTok{(wmat }\SpecialCharTok{*}\NormalTok{ (dold }\SpecialCharTok{{-}}\NormalTok{ aold) }\SpecialCharTok{\^{}} \DecValTok{2}\NormalTok{)}
    \ControlFlowTok{repeat}\NormalTok{ \{}
\NormalTok{      mmat }\OtherTok{\textless{}{-}} \SpecialCharTok{{-}}\NormalTok{aold }\SpecialCharTok{*}\NormalTok{ wmat }\SpecialCharTok{/}\NormalTok{ (dold }\SpecialCharTok{+} \FunctionTok{diag}\NormalTok{(n))}
      \FunctionTok{diag}\NormalTok{(mmat) }\OtherTok{\textless{}{-}} \SpecialCharTok{{-}}\FunctionTok{rowSums}\NormalTok{(mmat)}
\NormalTok{      bmat }\OtherTok{\textless{}{-}} \SpecialCharTok{{-}}\NormalTok{wmat }\SpecialCharTok{*}\NormalTok{ delta }\SpecialCharTok{/}\NormalTok{ (dold }\SpecialCharTok{+} \FunctionTok{diag}\NormalTok{(n))}
      \FunctionTok{diag}\NormalTok{(bmat) }\OtherTok{\textless{}{-}} \SpecialCharTok{{-}}\FunctionTok{rowSums}\NormalTok{(bmat)}
\NormalTok{      umat }\OtherTok{\textless{}{-}}\NormalTok{ ((}\DecValTok{1} \SpecialCharTok{{-}}\NormalTok{ sold) }\SpecialCharTok{*}\NormalTok{ vmat) }\SpecialCharTok{+}\NormalTok{ (sold }\SpecialCharTok{*}\NormalTok{ mmat)}
\NormalTok{      uinv }\OtherTok{\textless{}{-}} \FunctionTok{solve}\NormalTok{(umat }\SpecialCharTok{+} \DecValTok{1}\SpecialCharTok{/}\NormalTok{n) }\SpecialCharTok{{-}} \DecValTok{1}\SpecialCharTok{/}\NormalTok{n}
\NormalTok{      xnew }\OtherTok{\textless{}{-}}\NormalTok{ uinv }\SpecialCharTok{\%*\%}\NormalTok{ bmat }\SpecialCharTok{\%*\%}\NormalTok{ xold}
\NormalTok{      dnew }\OtherTok{\textless{}{-}} \FunctionTok{as.matrix}\NormalTok{(}\FunctionTok{dist}\NormalTok{(xnew))}
\NormalTok{      anew }\OtherTok{\textless{}{-}} \FunctionTok{sum}\NormalTok{(wmat }\SpecialCharTok{*}\NormalTok{ dnew)}
\NormalTok{      snew }\OtherTok{\textless{}{-}} \FunctionTok{sum}\NormalTok{(wmat }\SpecialCharTok{*}\NormalTok{ (delta }\SpecialCharTok{{-}}\NormalTok{ dnew) }\SpecialCharTok{\^{}} \DecValTok{2}\NormalTok{) }\SpecialCharTok{/} \FunctionTok{sum}\NormalTok{(wmat }\SpecialCharTok{*}\NormalTok{ (dnew }\SpecialCharTok{{-}}\NormalTok{ anew) }\SpecialCharTok{\^{}} \DecValTok{2}\NormalTok{)}
      \ControlFlowTok{if}\NormalTok{ (verbose) \{}
        \FunctionTok{cat}\NormalTok{(}
          \StringTok{"itel "}\NormalTok{,}
          \FunctionTok{formatC}\NormalTok{(itel, }\AttributeTok{format =} \StringTok{"d"}\NormalTok{),}
          \StringTok{"sold "}\NormalTok{,}
          \FunctionTok{formatC}\NormalTok{(sold, }\AttributeTok{digits =} \DecValTok{10}\NormalTok{, }\AttributeTok{format =} \StringTok{"f"}\NormalTok{),}
          \StringTok{"snew "}\NormalTok{,}
          \FunctionTok{formatC}\NormalTok{(snew, }\AttributeTok{digits =} \DecValTok{10}\NormalTok{, }\AttributeTok{format =} \StringTok{"f"}\NormalTok{),}
          \StringTok{"}\SpecialCharTok{\textbackslash{}n}\StringTok{"}
\NormalTok{        )}
\NormalTok{      \}}
      \ControlFlowTok{if}\NormalTok{ ((itel }\SpecialCharTok{==}\NormalTok{ itmax) }\SpecialCharTok{||}\NormalTok{ ((sold }\SpecialCharTok{{-}}\NormalTok{ snew) }\SpecialCharTok{\textless{}}\NormalTok{ eps)) \{}
        \ControlFlowTok{break}
\NormalTok{      \}}
\NormalTok{      sold }\OtherTok{\textless{}{-}}\NormalTok{ snew}
\NormalTok{      dold }\OtherTok{\textless{}{-}}\NormalTok{ dnew}
\NormalTok{      xold }\OtherTok{\textless{}{-}}\NormalTok{ xnew}
\NormalTok{      aold }\OtherTok{\textless{}{-}}\NormalTok{ anew}
\NormalTok{      itel }\OtherTok{\textless{}{-}}\NormalTok{ itel }\SpecialCharTok{+} \DecValTok{1}
\NormalTok{    \}}
    \FunctionTok{return}\NormalTok{(}\FunctionTok{list}\NormalTok{(}
      \AttributeTok{x =}\NormalTok{ xnew,}
      \AttributeTok{s =}\NormalTok{ snew,}
      \AttributeTok{d =}\NormalTok{ dnew,}
      \AttributeTok{b =}\NormalTok{ bmat,}
      \AttributeTok{m =}\NormalTok{ mmat,}
      \AttributeTok{w =}\NormalTok{ wmat,}
      \AttributeTok{a =}\NormalTok{ anew,}
      \AttributeTok{u =}\NormalTok{ umat,}
      \AttributeTok{itel =}\NormalTok{ itel}
\NormalTok{    ))}
\NormalTok{  \}}

\NormalTok{torgerson }\OtherTok{\textless{}{-}} \ControlFlowTok{function}\NormalTok{(delta, ndim) \{}
\NormalTok{  dd }\OtherTok{\textless{}{-}}\NormalTok{ delta }\SpecialCharTok{\^{}} \DecValTok{2}
\NormalTok{  rd }\OtherTok{\textless{}{-}} \FunctionTok{apply}\NormalTok{(dd, }\DecValTok{1}\NormalTok{, mean)}
\NormalTok{  rr }\OtherTok{\textless{}{-}} \FunctionTok{mean}\NormalTok{(dd)}
\NormalTok{  cc }\OtherTok{\textless{}{-}} \SpecialCharTok{{-}}\NormalTok{.}\DecValTok{5} \SpecialCharTok{*}\NormalTok{ (dd }\SpecialCharTok{{-}} \FunctionTok{outer}\NormalTok{(rd, rd, }\StringTok{"+"}\NormalTok{) }\SpecialCharTok{+}\NormalTok{ rr)}
\NormalTok{  ec }\OtherTok{\textless{}{-}} \FunctionTok{eigen}\NormalTok{(cc)}
\NormalTok{  xx }\OtherTok{\textless{}{-}}\NormalTok{ ec}\SpecialCharTok{$}\NormalTok{vectors[, }\DecValTok{1}\SpecialCharTok{:}\NormalTok{ndim] }\SpecialCharTok{\%*\%} \FunctionTok{diag}\NormalTok{(}\FunctionTok{sqrt}\NormalTok{(ec}\SpecialCharTok{$}\NormalTok{values[}\DecValTok{1}\SpecialCharTok{:}\NormalTok{ndim]))}
  \FunctionTok{return}\NormalTok{(xx)}
\NormalTok{\}}
\end{Highlighting}
\end{Shaded}

\section*{References}\label{references}
\addcontentsline{toc}{section}{References}

\phantomsection\label{refs}

\end{document}